\documentclass[letter,10pt]{article}
\usepackage{amssymb}
\usepackage{amsthm}
\usepackage{amsmath}
\usepackage{mathrsfs}
\usepackage{verbatim}
\usepackage{enumitem}

\title{On a class of optimal transportation problems with infinitely many marginals\footnote{The author was supported in part by a University of Alberta start-up grant.}}
\author{Brendan Pass\footnote{Department of Mathematical and Statistical Sciences, 632 CAB, University of Alberta, Edmonton, Alberta, Canada, T6G 2G1 pass@ualberta.ca.}}

\begin{document}

\maketitle

\begin{abstract}
We prove existence and uniqueness results for solutions to a class of optimal transportation problems with infinitely many marginals, supported on the real line.  We also provide a characterization of the solution with an explicit formula.  We then show that this result implies an infinite dimensional rearrangement inequality, and use it to derive upper bounds on solutions to parabolic PDE via the Feynman-Kac formula.  This result can be used to provide model independent bounds on the prices of certain derivatives in mathematical finance.  As an another application, we obtain refined phase-space bounds in quantum physics.
\end{abstract}

\section{Introduction}

In this paper we introduce and initiate the investigation of a novel problem - an optimal transportation problem of Monge Kantorovich type with infinitely many marginals.  Given measures (or marginals) $\mu_t$ on $\mathbb{R}$, for all $t \in [0,1]$, and a strictly convex function $h: \mathbb{R} \rightarrow \mathbb{R}$, we look for a stochastic process $X_t$ with continuous sample paths, such that law$X_t =\mu_t$ for all $t\in [0,1]$, which maximizes

\begin{equation*} \tag{$MK_{\infty}$}
E\Big(h\big(\int_{0}^1 X_t dt\big)\Big)
\end{equation*}

Our motivation in studying ($MK_{\infty}$) comes partially from an analogy with more familiar problems with finitely many marginals.  The (two marginal) optimal transportation problem of Monge and Kantorovich is to couple a pair of probability measures together as efficiently as possible.  This problem has many deep connections to a variety of areas of mathematics, such as geometry and PDE, and has found applications in many diverse fields, including economics, fluid mechanics, meteorology, and image processing \cite{V}\cite{V2}.  Given two Borel probability measures $\mu_0$ and $\mu_1$ on smooth manifolds $M_0$ and $M_1$, respectively, and a surplus function $b: M_0 \times M_1 \rightarrow \mathbb{R}$, the problem consists of finding the maximizer of the functional:

\begin{equation*} \tag{$MK_{2}$}
\int_{M_0 \times M_1}b(x_0,x_1)d\gamma
\end{equation*}
among measures $\gamma$ on $M_0 \times M_1$ which project to $\mu_0$ and $\mu_1$ under the canonical projections.  Under mild conditions, a solution $\gamma$ exists.  Under somewhat stronger conditions, one can show that this solution is unique and is concentrated on the graph $\{x_1=T(x_0)\}$ of a function over $x_0$; see Brenier \cite{bren}, Gangbo and McCann \cite{gm}, Gangbo \cite{g}, Caffarelli \cite{caf}, McCann \cite{m}, and Levin \cite{lev}.  In this case, we say that $T$ solves Monge's optimal transportation problem.  Let us note that the class of cost functions satisfying this condition includes $c(x_0,x_1) = h(x_0+x_1)$, where $h$ is strictly convex.

An interesting extension of $(MK_{2})$ is the problem of aligning several distributions with maximal efficiency.  That is, given $m$ Borel probability measures $\mu_{i}$, for $i =1,2,....,m$ on smooth manifolds $M_1,M_2,...,M_m$, respectively, and a surplus function $b: M_1 \times M_2 \times ...\times M_m \rightarrow \mathbb{R}$, maximizing

 \begin{equation*}\tag{$MK_m$}
\int_{M_0 \times M_1 \times ...\times M_m}b(x_1,x_2,...,x_m)d\gamma
\end{equation*}
among measures $\gamma$ on $M_1 \times M_2 \times ...\times M_m$ which project to the $\mu_i$ under the canonical projections, for $i=1,2,...,m$.

For our purposes, it will be useful to have an equivalent formulation of this problem using more probabilistic language.  From this perspective, we look for the $M_0 \times M_1 \times ...\times M_m$ valued random variable $(X_1,X_2,...,X_m)$, such that law$(X_i) = \mu_i$ for all $i=1,2,...,m$, which maximizes the \textit{expectation}:

 \begin{equation*}\tag{$MK_m$}
E(c(X_1,X_2,...,X_m))
\end{equation*}

As in the $m=2$ case, it is straightforward to prove the existence of a maximizer.  The literature on $(MK_m)$, however, is much more limited than that on $(MK_2)$.  The case that is most relevant to us here is when $M_i = \mathbb{R}^n$ for all $i=1,2,...,m$ and $c(x_1,x_2,...,x_m)=h(\sum_{i=1}^mx_i)$, for a strictly convex function $h$, this problem has been studied by Gangbo and Swiech \cite{gs} and Heinich \cite{h}, using a duality result of Kellerer \cite{K}; see also \cite{OR}\cite{KS}\cite{RU} and \cite{RU2}.  In this case, they proved that the optimal measure $\gamma$ is unique and that it is concentrated on the graph $\{x_1,f_2(x_2),....f_m(x_m)\}$ of a function over the first marginal, inducing what is often called a Monge solution.  Since then, a handful of additional results haven been obtained for other surplus functions, by Carlier \cite{C}, Carlier and Nazaret \cite{CN}, and the present author \cite{P}\cite{P1}.

It is natural to ask about the limit as the number of marginals, $m$, goes to infinity.  Here, instead of prescribing $m$ marginals, we prescribe a \textit{continuum} of marginals $\mu_t$ on $\mathbb{R}^n$, for $t \in [0,1]$, and rather than looking for a random variable $(X_1,...,X_m)$, we search for random \textit{curves} $X_t$, (ie, stochastic processes) that satisfy law($X_t) =\mu_t$ for each $t \in [0,1]$.  Heuristically, replacing the cost function $h(\sum_{i=1}^mx_i)$ with $h(\int_0^1x_tdt)$, we obtain problem $(MK_{\infty})$ as the limit of $(MK_m)$ as $m \rightarrow \infty$.

In this paper, we study $(MK_{\infty})$ but restrict our attention to the case $n=1$.  The machinery developed here does not generalize easily to the higher dimensional case; however, in a subsequent paper, we address the higher dimensional case for the quadratic surplus (that is, $h(\cdot) =|\cdot|^2$), using very different techniques \cite{P5}.  In the multi-marginal setting, when $n=1$, Carlier \cite{C} showed that, for a class of cost functions including strictly convex functions of the sum, the optimal maps $f_i:\mathbb{R} \rightarrow \mathbb{R}, i=2,3,...,m$ are all increasing. Our main result here (Theorem \ref{main}) is an infinite marginal analogue; there exists a unique optimal stochastic process $X_t$ and it is concentrated on an increasing family of curves.

When $h(x) = e^{x}$, the reader may recognize expressions of form $(MK_{\infty})$ from the Feynman-Kac representation of solutions to parabolic PDE (the Feynman-Kac expression is actually somewhat more complicated, but see the extension in subsection 3.2).  Our main result implies a \textit{stochastic rearrangement inequality} (Theorem \ref{reeq}); for any stochastic process, $X_t$, $E\Big(h\big(\int_{0}^1 X_t dt\big)\Big)$ can be bounded above by the same expression with the process $X_t$ replaced by a particularly simple process.  Therefore, this result can be used to derive upper bounds on the solutions to parabolic PDE which are straightforward to compute.  Let us briefly mention that upper bounds on expressions of the form $(MK_{\infty})$ can be also be found by Jensen's inequality, but that our bound is always tighter than the Jensen's inequality bound.  We discuss this in more detail in section 4.

Parabolic PDE arise frequently in various applied contexts, including mathematical finance and mathematical physics among many others.  In mathematical finance, for example, the price of a zero coupon bond maturing in one year is given by $E[e^{-\int_0^1 X_tdt}]$; where the stochastic process $X_t$ represents (stochastic) interest rate, given by some model.  As another example, the value of an Asian call option on a stock maturing in one year is $E(h(\int_0^1 X_tdt))$, where $X_t$ is the value of the stock price and $h(x) = \max(0,x-k)$ for some fixed $k$ (the strike price) is piecewise linear.  In some of these models, the expectation itself may be difficult to compute analytically, but the single time marginals can be readily determined and results in section 4 can be used to determine an upper bound on the interest rate simply by doing two 1-dimensional integrals.    

Let us add that, from a financial perspective, the potential interest in these inequalities does not only lie in estimating integrals that are difficult to compute; it can also be used provide information about certain phenomena that is less model dependent than standard financial methods.  Model independence in finance has recently seen quite a bit of attention.  The goal here is to understand the difference between derivative prices arising from models with the same marginals, as is outlined in \cite{bhlp}, \cite{hl} and \cite{ght}.  Note that in these papers, while the authors restrict to discrete time models, they also impose the condition that the process $X_t$ should be a martingale, which is natural in a financial context.  Our optimal process, $X_t$ for $(MK_{\infty})$, is not a martingale; nevertheless, it does provide an upper bound on the prices of appropriate derivatives arising from any model yielding the same marginals, and so may be relevant in their context.

Expressions of the form $E\Big(h\big(\int_{0}^1 X_t dt\big)\Big)$ also arise frequently in mathematical physics.  In this context, upper bounds on these expressions can be used to derive phase space bounds in quantum mechanics.    The best known example is perhaps Lieb's proof of the Cwickel-Lieb-Rosenbljum (CLR) bound, which itself is a special case of Lieb-Thirring inequalities, and can be used to derive the more general case \cite{lieb}.  The Lieb-Thirring inequalities are well known for their crucial role in Lieb and Thirring's famous proof of the stability of matter \cite{liebthirr}.  The general process for obtaining phase-space bounds is outlined nicely in the book by Simon \cite{simon}; very briefly, the key idea is to apply Jensen's inequality to expressions of the form $(MK_{\infty})$, then use Fubini's theorem to rearrange the upper bound into a desirable form.  We show that our result here provides tighter bounds on these quantities.

In the next section, we outline our main assumptions and make our basic definitions.  In section 3, we state and prove our main result.  In the fourth section, we use our main result to derive a rearrangement inequality for stochastic processes, which can then be used to derive upper bounds on solutions to parabolic PDE via the Feynman-Kac formulas.  In the final section, we describe a specific application of this inequality: new phase space bounds in quantum physics.
\section{Preliminaries and assumptions}
In this section, we outline our assumptions and preliminary definitions.  Many of these assumptions can be relaxed, at the cost of more complicated proofs; in subsection 3.2, for example, we discuss two extensions of this theorem, including a relaxation of the first assumption.

Let $\Pi$ be the set of all real valued, measurable stochastic processes $X_t$ with continuous sample paths such that law$X_t =\mu_t$ for all $t \in [0,1]$.  We will assume:
\begin{enumerate}
\item The function $h:\mathbb{R} \rightarrow \mathbb{R}$ is strictly convex and $C^2$ smooth.
\item Each $\mu_t$ is absolutely continuous with respect to Lebesgue measure.
\item The support of each $\mu_t$ is a (possibly unbounded) interval: spt$\mu_t =[a_t,b_t]$, where $a_t$ may be $-\infty$ and $b_t$ may be $\infty$.
\item The path of measures $t \mapsto \mu_t$ is weakly continuous.
\item The \textit{first moments} and \textit{h moments} of $\mu_t$ exist and are integrable in $t$.  That is, 
\begin{equation*}
\int_{0}^1 \big|\int_{\mathbb{R}}x d\mu_t(x)\big|dt <\infty \text{ ,  }\int_{0}^1 \big|\int_{\mathbb{R}}h(x) d\mu_t(x)\big|dt <\infty. 
\end{equation*}
\end{enumerate}

Note that, for any process $X_t \in \Pi$, applying Jensen's inequality twice and using Fubini's theorem yields:
 
 \begin{eqnarray*}
 h(\int_{0}^1 \int_{\mathbb{R}^n}x d\mu_t(x)dt) = h(\int_{0}^1 E(X_t) dt)& = &h(E(\int_{0}^1 X_t dt)) \\
 &\leq & E(h(\int_{0}^1 X_t dt))\\
 & \leq &E(\int_{0}^1h (X_t) dt) =\int_{0}^1E(h(X_t)) dt =\int_{0}^1 \int_{\mathbb{R}} h(x)d\mu_t(x)dt
 \end{eqnarray*}
 
Therefore, assumption 5 ensures that $E(h(\int_{0}^1 X_t dt))$ is finite for each competitor $X_t \in \Pi$.

Our last assumption requires a little more notation, which we develop now.  For each $\lambda \in (0,1)$, define the curve 

\begin{equation*}
x_{\lambda}(t) =F_t^{-1}(\lambda)
\end{equation*}
where $F_t^{-1}$ is the inverse of the cumulative distribution function $F_t$ of $\mu_t$.  Our final assumption can be interpreted as a regularity condition on the curve of measures $\mu_t$, expressed in terms of the curves $x_{\lambda}(t)$:
\begin{enumerate}[resume]
\item The Riemann sums $\sum_{i=1}^Nx_{\lambda} (\frac{i}{N})$ converge to $\int_{0}^{1}x_{\lambda}(t)dt$ as $N \rightarrow \infty$, uniformly in $\lambda$ on compact subsets of $(0,1)$.
\end{enumerate}

Note that weak continuity of the family $\mu_t$ (assumption 4) implies that the curve $t \rightarrow x_{\lambda}(t)$ is continuous and hence Riemann integrable.  The Riemann sums above then converge to the integral, but the convergence is not necessarily uniform in $\lambda$.
 
Now, we define a particular stochastic process in $\Pi$, which, as we will prove in the next section, is actually optimal for $(MK_{\infty})$. 

Our underlying probability space will be the interval $(0,1)$, with Lebesgue measure.  Define the process 
\begin{equation}\label{process}
X^{opt}_t=X^{opt}_t(\lambda) =x_{\lambda}(t), \text{ for } \lambda \in [0,1].  
\end{equation}
That is, the sample paths are the curves $t \mapsto x_{\lambda}(t)$, distributed uniformly in $\lambda$.  Note that these sample paths are increasing in $\lambda$: whenever $\lambda \geq \overline{\lambda}$, $x_{\lambda}(t) >\geq x_{\overline{\lambda}}(t)$ for all $t$.  In addition, it is clear from the construction that law$X^{opt}_t = \mu_t$.

Next, we define some functions which will be useful in the proof of our main result.  For each $t \in [0,1]$, we define a function $g_t:\mathbb{R} \rightarrow \mathbb{R}$ by  
\begin{equation*}
g_t(x)= h'(\int_{0}^{1}x_{\lambda}(r)dr),
\end{equation*}
where $\lambda = \mu_t((-\infty,x)) =F_t(x)$.  We then set

\begin{equation*}
G_t(x) = \int_{x_{\frac{1}{2}}(t)}^x g_t(y)dy
\end{equation*}
As $h$ is convex, $x \mapsto h'(x)$ is increasing.  As $x \mapsto F_t(x) =\lambda$ and $\lambda \mapsto \int_0^1x_{\lambda}(t)dt$ are also increasing, it follows that $x \mapsto g_t(x)$ is an increasing function of $x$ and so $x \mapsto G_t(x)$ is convex.

Finally, as the function $\lambda \mapsto \int_0^1x_{\lambda}(t)dt$ is monotone, its (possibly infinite) limit as $\lambda$ increases to 1 exists: denote $I_{max} = \lim_{\lambda \rightarrow 1^-}\int_0^1x_{\lambda}(t)dt$.  Similarly, define $I_{min} = \lim_{\lambda \rightarrow 0^+}\int_0^1x_{\lambda}(t)dt$.

\section{Main Result}

\subsection{Statement and proof of main result}
Our main result is the following.
\newtheorem{main}{Theorem}[subsection]

\begin{main} 
\label{main}
 The process $X^{opt}_t$ defined in equation (\ref{process}) is the unique maximizer in $(MK_{\infty})$.
\end{main}
The proof requires several Lemmas.
 
\newtheorem{dual}[main]{Lemma}
\begin{dual}\label{dual}
For any process $X_t \in \Pi$, $E (\int_{0}^{1}G_t(X_t)dt)  =\int_{0}^{1}\int_{\mathbb{R}}G_t(x)d\mu_t(x)dt $
\end{dual}
\begin{proof}
The result is a simple application of Fubini's theorem:
 \begin{eqnarray*}
  E (\int_{0}^{1}G_t(X_t)dt)  & = & \int_{0}^{1}E(G_t(X_t))  dt\\
& =& \int_{\mathbb{R}} \int_{0}^{1}G_t(x) d\mu_t(x) dt
 \end{eqnarray*}
\end{proof}
\newtheorem{ineq}[main]{Lemma}
\begin{ineq}\label{ineq}
 For each $\lambda \in (0,1)$, the curve $x_{\lambda}(\cdot)$ maximizes the functional 
\begin{equation}
 \mathcal{F}(\tilde{x}(\cdot)) = h(\int_0^1 \tilde{x}(t)dt) - \int_0^1G_t(\tilde{x}(t))dt \label{funct}
\end{equation}
among all continuous curves $\tilde{x}(\cdot)$ on $[0,1]$, with $\tilde{x}(t) \in (a_t,b_t)$ for all $t$.
\end{ineq}

\begin{proof}

Our strategy will be to approximate these integrals by Riemann sums, and exploit known results on the corresponding optimal transportation problem with finitely many marginals.
For a positive integer $N$, set $t_i = \frac{i}{N}$, for $i=1,2,...,N$ and consider the $N$-marginal optimal transportation problem with marginals $\mu_{t_i}$ and surplus function 
\begin{equation*}
b(x_{t_1}, x_{t_2},...,x_{t_N})=h(\frac{\sum_{i=1}^Nx_{t_i}}{N}).
\end{equation*}
Note that this surplus function satisfies the strictly two monotone condition in \cite{C} (essentially that $\frac{\partial^2 b}{\partial x_{t_i}\partial x_{t_j}}>0$ for all $i \neq j$); by a result of Carlier \cite{C} (extended to unbounded domains in \cite{P4}), the solution is concentrated on the set 
\begin{equation*}
 \{(x_{\lambda}(t_1), x_{\lambda}(t_2),...,x_{\lambda}(t_N)): \lambda \in (0,1)\}.
\end{equation*}
It is well known that there exists a solution $(\phi^N_{t_1}, \phi^N_{t_2},...,\phi^N_{t_N})$ to the \textit{Kantorovich dual problem }\cite{K}.  These functions satisfy:

\begin{equation*}
\sum_{i=1}^N\phi^N_{t_i}(x_{t_i}) \geq h(\frac{\sum_{i=1}^Nx_{t_i}}{N})
\end{equation*} 
with equality on the set $\{(x_{\lambda}(t_1), x_{\lambda}(t_2),...,x_{\lambda}(t_N)): \lambda \in (0,1)\}$.

Furthermore, due to the convexity of $h$, the functions $\phi^N_{t_i}$ are all convex, and due to assumptions 2 and 3, they are differentiable and satisfy the first order condition
\begin{equation*}
\frac{1}{N}h\prime(\frac{\sum_{i=1}^N x_{\lambda}(t_i)}{N}) =  \phi^{N}_{t_i} \prime (x_{\lambda}(t_i)), 
\end{equation*}
for all $\lambda \in (0,1)$.  

Up to an additive constant, we can assume that each $\phi^N_{t_i}(x_{\frac{1}{2}}(t_i)) = 0$

Now, let $t \rightarrow \tilde{x}(t)$ be any continuous curve such that $\tilde{x}(t) \in (a_t,b_t)$ for all $t \in [0,1]$. Then we have, for any $\lambda$:

\begin{equation} \label{finite}
h(\frac{\sum_{i=1}^N\tilde {x}(t_i)}{N}) - \sum_{i=1}^N \phi_{t_i}^N(\tilde{x}(t_i)) \leq h(\frac{\sum_{i=1}^N x_{\lambda}(t_i)}{N}) - \sum_{i=1}^N \phi_{t_i}^N(x_{\lambda}(t_i)) 
\end{equation}

For each $s \in [0,1]$, set $\lambda_s =F_s(\tilde{x}(s))$.  We then have $\lambda_s \in (0,1)$ and $x_{\lambda_s}(s) = \tilde{x}(s)$.  By assumptions 2 and 4, the map $s \mapsto \lambda_s$ is continuous.  Therefore, there exists some $0<c\leq d<1$ such that $c\leq \lambda_s \leq d$ for all $0\leq s\leq1$.

 Therefore, by the monotonicity of $\lambda \mapsto x_{\lambda}(t)$, we have $x_{c}(t) \leq x_{\lambda_s}(t) \leq x_{d}(t)$, and so, for all $s \in [0,1]$,

\begin{equation*}
I_c:=\int_0^1 x_{c}(t) dt \leq \int_0^1x_{\lambda_s}(t)dt \leq\int_0^1  x_{d}(t)dt=:I_d
\end{equation*}

Therefore,  $s \mapsto \int_0^1 x_{\lambda_s}(t)dt$ is bounded above and below; choose $K$ such that $\big|h'(x)\big| \leq K$ whenever $I_c-1\leq x \leq I_d+1$

Now, for each point $x_{t_i} \in (a_{t_i},b_{t_i})$ such that $\lambda :=F_{t_i}(x_{t_i}) \in [c,d]$, we have
\begin{eqnarray}
|\phi^N_{t_i}\prime(x_{t_i}) - \frac{1}{N} g_{t_i}(x_{t_i})|& =& \frac{1}{N}\big|h'(\frac{\sum_{j=1}^Nx_{\lambda}(t_j)}{N}) -h'(\int_{0}^1 x_{\lambda}(t)dt)\big|\\
& \leq & K  \frac{1}{N}\big|\frac{\sum_{j=1}^Nx_{\lambda}(t_j)}{N} -\int_{0}^1 x_{\lambda}(t)dt\big| \label{derivative}
\end{eqnarray}
 
Choose $\epsilon >0$. By uniform convergence of the Riemann sums of $x_{\lambda}(t)$ (assumption 6), we can choose $N$ so large that we have 
\begin{equation}
\big|\frac{\sum_{i=1}^Nx_{\lambda}(t_i)}{N} -\int_{0}^1 x_{\lambda}(t)dt\big| < \epsilon \label{riesums}
\end{equation}
for all $\lambda \in [c,d]$.

Now, the continuous function $t \mapsto \tilde{x}(t) -x_{\frac{1}{2}}(t)$ is bounded on $[0,1]$: 
\begin{equation}
|\tilde{x}(t)-x_{\frac{1}{2}}(t)|\leq M, \label{xbound}
\end{equation}
and so integrating $\phi^N_{t_i}\prime (x_{t_i})$ and $g_{t_i}(x_{t_i})$ from $x_{\frac{1}{2}}(t_i)$ to $\tilde{x}(t_i)$ and combining inequalities (\ref{derivative}), (\ref{riesums}), and (\ref{xbound}) yields 
\begin{eqnarray*}
 |\phi^N_{t_i}(\tilde{x}(t_i)) - \frac{1}{N}G_{t_i}(\tilde{x}(t_i))|= &\leq  &  \int_{x_{\frac{1}{2}}(t_i)}^{\tilde{x}(t_i)}|\phi^{N\prime}_{t_i}(x) - \frac{1}{N} g_{t_i}(x)| dx \\
 &\leq& KM \frac{1}{N} \epsilon
\end{eqnarray*}

We then have 

\begin{equation*}
|\sum_{i=1}^N\phi^N_{t_i}(\tilde{x}(t_i))  - \sum_{i=1}^N\frac{1}{N}G_{t_i}(\tilde{x}(t_i))| \leq KM \epsilon
\end{equation*}

Now, the term $\sum_{i=1}^N\frac{1}{N}G_{t_i}(\tilde{x}_{t_i})$ converges to $\int_{0}^{1}G_t(\tilde{x}(t))dt$ as $N \rightarrow \infty$ and so the above inequality implies that $\sum_{i=1}^N\phi^N_{t_i}(\tilde{x}_{t_i})$ converges to $\int_{0}^{1} G_t(\tilde{x}(t))dt$ as well.  Similarly, $\sum_{i=1}^N\phi^N_{t_i}(x_{\lambda}(t_i))$ converges to $\int_{0}^{1} G_t(x_{\lambda}(t))dt$.  As $h(\frac{\sum_{i=1}^N\tilde {x}_{t_i}}{N})$ converges to $h(\int_{t=0}^1\tilde{ x}(t)dt)$ and $h(\sum_{i=1}^Nx_{\lambda}(t_i))$ converges to $h(\int_{t=0}^1 x_{\lambda}(t)dt)$ the result now follows by taking the limit of (\ref{finite}).

\end{proof}
\newtheorem{ineq2}[main]{Lemma}
\begin{ineq2}\label{ineq2}
 For each $\lambda \in (0,1)$, the curve $x_{\lambda}(\cdot)$ maximizes the functional 
\begin{equation}
 \mathcal{F}(\tilde{x}(\cdot)) = h(\int_0^1 \tilde{x}(t)dt) - \int_0^1G_t(\tilde{x}(t))dt \label{funct}
\end{equation}
among all continuous curves $\tilde{x}(t)$ on $[0,1]$, with $\tilde{x}(t) \in spt(\mu_t)$ for all $t$.
\end{ineq2}

\begin{proof}
By the preceding lemma, we only need to consider the case where $\tilde{x}(t) = b_t$ or $a_t$ for some $t$.  We will assume here that $\tilde{x}(t) = b_t$ for some $t$, but $a_t <\tilde{x}(t)$ for all $t$ (the proofs of the other cases are similar).

First, suppose that $\lim_{x \rightarrow I_{max}^+} h'(x) >0$.  This implies that $\lim_{x \rightarrow b_t}g_t(x) >0$ for each $t$.  Then, for $\epsilon >0$, consider the set

\begin{equation*}
A_{\epsilon} = \{t: \tilde{x}(t) -\epsilon \in (a_t,b_t)\} \cap \{t: g_t(\tilde{x}(t) -\epsilon)>0\}
\end{equation*}
As an open subset of $[0,1]$, $A_{\epsilon}$ can be written as a countable union of disjoint intervals, $A_{\epsilon} =\cup_{i}I_i$.  For each $i$, let $q_i(t): I_i \rightarrow \mathbb{R}$ be a continuous function such that $q_i(t) =0$ on the boundary of $I_i$ and $0< q_i(t) <1$ on the interior of $I_i$.  We define the function

\begin{eqnarray*}
\tilde{x}^{\epsilon}(t) &= &\tilde{x}(t) -\epsilon q_i(t) , \text{ if } t \in I_i \text{ for some } i\\
\tilde{x}^{\epsilon}(t) &= &\tilde{x}(t), \text{ otherwise.}
\end{eqnarray*}

Note that $\tilde{x}^{\epsilon}(t) \leq \tilde{x}(t)$ everywhere, and the inequality is strict for $t \in A_{\epsilon}$.  As  $x \mapsto g_t(x)$ is monotone increasing and  $g_t(\tilde{x}(t) -\epsilon)>0$ on $A_{\epsilon}$, we have $g_t(x) >0$ for $ \tilde{x}(t) -\epsilon <x<\tilde{x}(t)$.  Thus, $\frac{dG_t}{dx}(x)=g_t(x)>0$, which implies $G_t(\tilde{x}^{\epsilon}(t))$ increases as $\epsilon$ decreases. Furthermore, $|\tilde{x}^{\epsilon}(t) - \tilde{x}(t)| \leq \epsilon$, so $\tilde{x}^{\epsilon}(t)  \rightarrow \tilde{x}(t)$ as $\epsilon \rightarrow 0$.  By continuity, $G_t(\tilde{x}^{\epsilon}(t))  \rightarrow G_t(\tilde{x}(t))$.

As the points where $\tilde{x}_t =b_t$ all lie in $A_{\epsilon}$, we have $a_t<\tilde{x}^{\epsilon}(t) <b_t$ everywhere. Therefore, Lemma \ref{ineq}
 
 \begin{equation*}
 h(\int_0^1 \tilde{x}^{\epsilon}(t)) - \int_0^1G_t(\tilde{x}^{\epsilon}(t))dt \leq  h(\int_0^1 x_{\lambda}(t)) - \int_0^1G_t(x_{\lambda}(t))dt
\end{equation*}

Taking the limit as $\epsilon \rightarrow 0$, using the monotone convergence theorem on $\int_0^1 \tilde{x}^{\epsilon}(t)dt$ and $\int_0^1G_t(\tilde{x}^{\epsilon}(t))dt$ and the continuity of $h$ yields the desired result.

Finally, if  $\lim_{x \rightarrow I_{max}^+} h'(x) \leq 0$, then $g_t(x)<0$ for all $x \in (a_t,b_t)$.  A similar argument to the above implies that $\epsilon \mapsto G_t(\tilde{x}^{\epsilon}(t))$ as $\epsilon$ decreases and applying the monotone convergence theorem again yields the desired result.

\end{proof}
\newtheorem{stineq}[main]{Lemma}
\begin{stineq}\label{stineq}
The curves $x_{\lambda}(t)$ are the \emph{only} maximizers of the functional  $\mathcal{F}(\tilde{x}(\cdot))$  in Lemma \ref{ineq2}.
\end{stineq}
\begin{proof}
To show that no other continuous curve $\tilde{x}$ can be a maximizer, we use a variational argument.  Assume that $\tilde{x}(\cdot)$ is a curve such that $\tilde{x}(t) \in spt(\mu_t)$ for all $t$,  maximizing (\ref{funct}).  Let $y(t)$ be a continuous function such that $\tilde{x}(t) + \epsilon y (t)) \in spt(\mu_t)$ for sufficiently small $\epsilon>0$.  We then have:

\begin{eqnarray}
0>\frac{dF( \tilde{x}(t) + \epsilon y (t))}{d\epsilon^{+}}\Big|_{\epsilon=0}&=&\frac{d}{d\epsilon^+}\Big|_{\epsilon=0}\Big[h\Big(\int_0^1 \tilde{x}(t) +\epsilon y (t)\Big)dt - \int_0^1G_t(\tilde{x}(t)+\epsilon y (t))dt\Big]\nonumber \\
&=&h'\Big(\int_0^1 \tilde{x}(t)dt\Big)\int_0^1y(t)dt - \int_0^1g_t(\tilde{x}(t))y (t)dt\nonumber\\
&=&\int_0^1 \Big[h'\Big(\int_0^1 \tilde{x}(r)dr\Big) - g_t(\tilde{x}(t))\Big]y (t)dt \label{decrease}
\end{eqnarray}
Note that we have used Lebesgue's dominated convergence theorem on the last term.

Now, choose $\lambda$ such that $\int_0^1x_{\lambda}(s)ds = \int_0^1\tilde{x}(s)ds$; we will show that $x_{\lambda}(t) = \tilde{x}(t)$ for all $t$.  The proof is by contradiction; assume $x_{\lambda}(t) \neq \tilde{x}(t)$ for some fixed $t$.  Suppose $x_{\lambda}(t) > \tilde{x}(t)$ (the case $x_{\lambda}(t) < \tilde{x}(t)$ is similar and is omitted).  Then by continuity, $x_{\lambda}(s) > \tilde{x}(s)$ in some neighbourhood $U$ of $t$.  By monotonicity of $x \mapsto g_s(x)$, this implies that $g_s(x_{\lambda}(s)) > g_s(\tilde{x}(s))$ in $U$ 

Now, as $ \tilde{x}(s) <x_{\lambda}(s) \leq b_s$ on $U$, we may choose $y(s)>0$ on $U$ and $y(s) = 0$ elsewhere and have $\tilde{x}(s) + \epsilon y(s) \in spt(\mu_s)$ everywhere.  Then we have, for $s \in U$
\begin{eqnarray*}
\Big[h'\Big(\int_0^1 \tilde{x}(r)dr\Big) - g_s(\tilde{x}(s))\Big]y (s) &=&\Big[h'\Big(\int_0^1 x_{\lambda}(r) dr\Big) -  g_s(\tilde{x}(s))\Big]y (s) \\
& >& \Big[h'\Big(\int_0^1 x_{\lambda}(r) dr\Big) - g_s(x_{\lambda}(s) )\Big]y (s) \\
& =& \Big[h'\Big(\int_0^1 x_{\lambda}(r) dr\Big) - h'\Big(\int_0^1 x_{\lambda}(r) dr\Big)\Big]y (s)\\
&=& 0
\end{eqnarray*}
As $y(s)=0$ outside of $U$, this yields 
\begin{equation*}
\int_0^1 \Big[h'\Big(\int_0^1 \tilde{x}(r)dr\Big) - g_s(\tilde{x}(s))\Big]y (s)ds >0
\end{equation*}

This contradicts inequality (\ref{decrease}), completing the proof.
 
\end{proof}

We are now in a position to prove Theorem \ref{main}.

\begin{proof}

It is clear from the definition that $X_t^{opt} \in \Pi$; that is, that law$X^{opt}_t = \mu_t$.  
Now,  by Lemma \ref{ineq2}, given any process $Y_t \in \Pi$ with single time marginals $\mu_t$, we have, for each individual sample path and each $\lambda$,

\begin{eqnarray*}
h(\int_{0}^1 Y_tdt) & \leq & h(\int_0^1 x_{\lambda}(t)) - \int_0^1G_t(x_{\lambda}(t))dt +\int_0^1G_t(Y_t)dt)\\
&=&M +\int_0^1G_t(Y_t)dt
\end{eqnarray*}
where $M$ denotes the minimal value of the functional $\mathcal{F}$ in Lemma \ref{ineq2}.  Note that we have equality if and only if the sample path $Y_t =x_{\lambda}(t)$ for some $\lambda$, by Lemma \ref{stineq}. Taking expectations and using Lemma \ref{dual} yields 
\begin{eqnarray*}
E(h(\int_{0}^1 Y_t dt)) &\leq& M  +E(\int_0^1G_t(Y_t)dt)\\
&=&M+\int_0^1\int_{\mathbb{R}}G_t(x)d\mu_t(x)dt.
\end{eqnarray*}
and we have equality if and only if $Y_t = x_{\lambda}(t)$ for some $\lambda$, almost surely.  Now, if we interpret a stochastic process as a measure on the set of continuous functions, it is clear that only one measure with marginals $\mu_t$ can be supported on the set $\{x_{\lambda}(\cdot): \lambda \in [0,1]\}$; namely, the optimizer $X^{opt}_t$.  This completes the proof.
\end{proof}

\subsection{Extensions of the main theorem}
In this subsection we present two extensions of our main theorem, designed specifically with potential applications in mind.

The first extension is a relaxation of the smoothness and strict convexity of $h$. One motivation for this result is to find an upper bound on the price of an Asian option, whose value is given by $E[h(\int_0^1X_tdt)]$, for a particular process $X_t$ with $h(x) =\max(0, ax-b)$ (ie, a non smooth, convex but not strictly convex surplus function).  In this case, we potentially lose uniqueness in our main theorem.
 
\newtheorem{extension2}{Theorem}[subsection]
\begin{extension2}
Suppose that $h$ is convex.  Assume conditions 2-6, and in addition assume that the \emph{second moments} of $\mu_t$ exist and are integrable, that is, $\int_0^1\int_{\mathbb{R}}x^2d\mu_t(x) dt< \infty$.  Then the process $X^{opt}_t$ defined in section 2 is optimal for $(MK_{\infty})$.

\end{extension2}

\begin{proof}
Let $h_n$ be a sequence of smooth, strictly convex functions such that $h_n(x) \geq h_{n+1}(x)$, $h_n(x) \rightarrow h(x)$  pointwise and 
\begin{equation*}h(x) \leq h_n(x) \leq h(x) +1 +x^2,
\end{equation*}
Then $h_n$ satisfies assumptions 1 and 5, and therefore, for any $Y_t \in \Pi$ we have, by Theorem \ref{main}

\begin{equation*}
E\Big(h_n\big(\int_{0}^1 Y_t dt\big)\Big) \leq E\Big(h_n\big(\int_{0}^1 X^{opt}_t dt\big)\Big)
\end{equation*}

Now, by the monotone convergence theorem, $E\Big(h_n\big(\int_{0}^1 Y_t dt\big)\Big)$ converges to $E\Big(h\big(\int_{0}^1 Y_t dt\big)\Big)$, and similarly $E\Big(h_n\big(\int_{0}^1 X^{opt}_t dt\big)\Big)$ converges to $E\Big(h\big(\int_{0}^1 X^{opt}_t dt\big)\Big)$, implying the desired result.
\end{proof}

Our second extension concerns coupling a stochastic process $X_t$ with prescribed single time marginals and a real valued random variable $Z$ with a prescribed law, to maximize $E[h(\int_0^1X_tdt)Z]$.  The motivation here is the Feynman-Kac formula; expressions of this form represent solutions to parabolic PDE and in the next section, we use our extension here to prove upper bounds on these solutions.   Our extension here holds when $Z$ is positive almost surely.  If $Z$ is merely non-negative almost surely, the process we construct is still optimal, but we possibly lose uniqueness.  

Let $\mu_t$ be a curve of measures, satisfying the conditions in the second section.  Let $\nu$ be a measure on the upper half line $\mathbb{R}_{+} = \{y >0\}$, absolutely continuous with respect to Lebesgue measure.  Then consider the problem of finding a pair $(X_t, Z)$, where $X_t$ is a continuous stochastic process, with law$X_t =\mu_t$ and $Z$ a random variable with law$Z =\nu$, maximizing

\begin{equation*} \tag{$MK_{\infty+1}$}
E\Big(h\big(\int_{0}^1 X_t dt\big)Z\Big)
\end{equation*}

Recall the stochastic process $X_t^{opt}$ from the section 2 and define the random variable $Z^{opt}=Z^{opt}(\lambda)$ on the same underlying probability space $[0,1]$ implicitly by
\begin{equation*}
 \nu(-\infty,Z^{opt}(\lambda)) =\lambda
\end{equation*}
Then the pair $(X^{opt}_t, Z^{opt})$ satisfies law$X^{opt}_t=\mu_t$ and law$Z^{opt}=\nu$.  We also define 

\begin{equation*}
\overline{g}(z)= h'(\int_{0}^{t}x_{\lambda}(t)dt)z_{\lambda},
\end{equation*}
where $\lambda = \nu((-\infty,z_{\lambda}))$.  We then set

\begin{equation*}
\overline{G}(z) = \int_{z^{1/2}}^z g_t(y)dy
\end{equation*}
\newtheorem{extension1}[extension2]{Theorem}

\begin{extension1}\label{extension1}
Assume conditions 1-6 from section 2 and that the measure $\nu$  on $(0,\infty)$ is absolutely continuous with respect to Lebesgue measure.  Then the pair $(X^{opt}_t,Z^{opt})$ is the unique optimizer for ($MK_{\infty+1}$).
\end{extension1}
\begin{proof}
The proof is similar to the proof of the main theorem, and is only sketched here.  The key adjustment is in Lemma \ref{ineq}, where we replace $h(\sum_{i=1}^Nx_{t_i})$ with $h(\sum_{i=1}^Nx_{t_i})z$, where $z>0$.  This surplus function is again strictly two monotone, and so the result of Carlier \cite{C} applies to the $N+1$ marginal optimal transportation problem with this surplus function and marginals $\mu_{t_1},\mu_{t_2},...\mu_{t_N},\nu$. Letting $\phi_{t_1}^N,\phi_{t_2}^N,...\phi^N_{t_N},u^N$ solve the dual problem, we obtain, for any curve $\tilde{x}(t)$,  point $\tilde{z}>0$ and $\lambda \in (0,1)$:

 \begin{equation*} 
h(\frac{\sum_{i=1}^N\tilde {x}(t_i)}{N})\tilde{z} - \sum_{i=1}^N \phi_{t_i}^N(\tilde{x}(t_i)) -u^N(\tilde{z}) \leq h(\frac{\sum_{i=1}^N x_{\lambda}(t_i)}{N})z_{\lambda} - \sum_{i=1}^N \phi_{t_i}^N(x_{\lambda}(t_i)) -u^N(z_{\lambda}).
\end{equation*}
From the first order conditions arising in the $N+1$ marginal problem, we have $u^{N\prime} (z) = h^{\prime} (\frac{\sum_{i=1}^N x_{\lambda}(t_i)}{N})z_{\lambda}$, where $\lambda = \nu((-\infty,z_{\lambda}))$.  Much like in the proof of Lemma \ref{ineq}, one can use this to show that as $N \rightarrow \infty$, $u^N(\tilde{Z}) \rightarrow \overline{G}(\tilde{Z})$.   As in the proof of Lemma \ref{ineq}, we can take the limit as $N \rightarrow \infty$ and obtain 
 \begin{equation*} 
h(\int_0^1 \tilde{x}(t)dt)\tilde{y} -  \int_0^1\overline{G}_t(\tilde{x}(t))dt -u(\tilde{y}) \leq h(\int_0^1x_{\lambda}(t)dt)\tilde{y} - \int_0^1\overline{G}_t(x_{\lambda}(t))dt -u(y_{\lambda}).
\end{equation*}

Here, $\overline{G}_t: = \int_{x_{\frac{1}{2}}(t)}^x g_t(y)dy$ is defined as in section 2, except that we take $\overline{g}_t(x)= h'(\int_{0}^{1}x_{\lambda}(r)dr)z_{\lambda}$, for $\lambda = F_t(x)$, whereas in section 2 we took $g(x)= h'(\int_{0}^{1}x_{\lambda}(r)dr)$.  

The remainder of the proof is analogous to the proof of the main theorem and preceding lemmas in subsection 3.1 and is omitted.
\end{proof}

\section{Rearrangement inequalities and applications}

\subsection{A rearrangement inequality}
Carlier observed that a solution to the multi-marginal problem in $1$-dimension yields a refinement of the Hardy-Littlewood inequality \cite{C}.  Here, we use our main theorem to develop an analogous infinite dimensional rearrangement inequality of Hardy-Littlewood type.  Given a stochastic process $Y_t$, set $\mu_t=$law$Y_t$.  We can then find the process $X^{opt}_t$ corresponding to these marginals, as in section 2 (that is, a monotone rearrangement of $Y_t$) and, assuming the marginals and convex $h$ satisfy assumptions 1-6, Theorems \ref{main} and \ref{extension1} yield the following result, which can be viewed as an infinite dimensional version of the Hardy-Littlewood inequality
\newtheorem{reeq}{Proposition}[subsection]
\begin{reeq}\label{reeq}
Let $Y_t$ be a stochastic process whose marginals $\mu_t =$law$(Y_t)$ satisfy conditions 1-6.  Let $Z$ be a random variable on $\mathbb{R}^+$ and assume $\nu =$law$(Z)$ is absolutely continuous with respect to Lebesgue.  Let $(X^{opt}_t, Z^{opt})$ be the optimal process corresponding to these marginals, defined in section 2 ($X^{opt}_t$) and subsection 3.2 ($Z^{opt}$). We then have
\begin{eqnarray*} 
E(h(\int_{0}^{1}Y_tdt)) &\leq& E(h(\int_{0}^{1}X^{opt}_tdt)) \\
&=&\int_{0}^{1}h(\int_{0}^{1}x_{\lambda}(t)dt)d\lambda
\end{eqnarray*} 
and
\begin{eqnarray*} 
E(h(\int_{0}^{1}Y_tdt)Z) &\leq& E(h(\int_{0}^{1}X^{opt}_tdt)Z^{opt}) \\
&=&\int_{0}^{1}h(\int_{0}^{1}x_{\lambda}(t)dt)z_{\lambda}d\lambda
\end{eqnarray*} 
\end{reeq}

It is interesting to note that Jensen's inequality also yields upper bounds on quantities of the form $E(h(\int_{0}^{1}Y_tdt))$, and these upper bounds are often used in finance and physics (see the examples in the subsequent section).  Recall that for each sample path $Y_t$ of a stochastic process and convex $h$, Jensen's inequality says:

\begin{equation*}
h(\int_{0}^1Y_tdt) \leq \int_0^1 h(Y_t)dt
\end{equation*}
so that, for any stochastic process $Y_t$, we have:
\begin{equation*}
E(h(\int_{0}^1Y_tdt)) \leq E(\int_0^1 h(Y_t)dt)
\end{equation*}

We now prove our bound (\ref{reeq}) is \textit{tighter} than the bound coming from Jensen's inequality.
\newtheorem{beatjen}[reeq]{Proposition}
\begin{beatjen}\label{beatjen}
Given a stochastic process $Y_t$, let $X_t^{opt}$ be the optimal process from section 2, with marginals $\mu_t = $law$(Y_t)$.  Then:
\begin{equation*}
E(h(\int_{0}^1X^{opt}_tdt))\leq E(\int_0^1 h(Y_t)dt)
\end{equation*}

\end{beatjen}
\begin{proof}
We apply Jensen's inequality to $E(h(\int_{0}^1X^{opt}_tdt))$, and use Fubini's theorem on the resulting expression, yielding 

\begin{eqnarray*}
E(h(\int_{0}^1X^{opt}_tdt)) &\leq& E(\int_0^1 h(X^{opt}_tdt))\\
&=&\int_0^1E( h(X^{opt}_t))dt\\
&=&\int_0^1 \int_{\mathbb{R}}h(x)d\mu_t(x))dt\\
&=&\int_0^1 \int_{\mathbb{R}}h(y)d\mu_t(y))dt\\
&=&\int_0^1E( h(Y_t))dt\\
&=&E(\int_0^1 h(Y_t)dt)\\
\end{eqnarray*}
where we have used the fact that the law$(Y_t)$ = law$(X^{opt}_t)$ =$\mu_t$.
\end{proof}

Also of note is that, by applying Jensen's inequality to the expectation $E$ rather than the integral $\int_0^1$, we obtain a \textit{lower} bound on $E(h(\int_{0}^{1}Y_t))$, which by Fubini's theorem, depends only on the marginals law$(Y_t)$.  Combining our main theorem with Jensen's inequality in this way, then, gives upper and lower bounds for $E(h(\int_{0}^{1}Y_t))$ which any process $Y_t$ with marginals $\mu_t$ must satisfy.

\subsection{Example}
We illustrate the rearrangement inequality in Proposition \ref{reeq} with an example.  In this example, the convex function $h(x) =e^{x}$ is the exponential function.

\newtheorem{brownmot}{Proposition}[subsection]
\begin{brownmot}
Let $W_t$ be Brownian motion starting at $W_0=0$ and $\mu_t =$law$(W_t)$, so that $d\mu_t(x)=\frac{e^{\frac{-x^2}{2t}}}{\sqrt{2\pi t}}dx$.  Then, for $s>0$,

\begin{eqnarray*}
E(e^{\int_0^s W_t dt}) &=& e^{\frac{s^3}{6}} \\
E(e^{(\int_0^s X_t^{opt}dt)})&=& e^{\frac{2s^3}{9}}\\
 E(\frac{1}{s}\int_{0}^se^{sW_t}dt) &=& \frac{2}{s^3}[e^{\frac{s^3}{2}}-1]
\end{eqnarray*}
\end{brownmot}

After rescaling the interval $[0,1]$ to $[0,s]$, these three quantities represent, respectively, $E(e^{\int_0^s W_t dt})$, the upper bound from Proposition \ref{reeq} (see the Remark below), and the upper bound from Jensen's inequality, respectively.  It is straightforward to verify directly that $e^{\frac{s^3}{6}} \leq e^{\frac{2s^3}{9}} \leq \frac{2}{s^3}[e^{\frac{s^3}{2}}-1]$, as is implied by Propositions \ref{reeq} and \ref{beatjen}.
\newtheorem{dirac}[brownmot]{Remark}
\begin{dirac}
Strictly speaking, our main theorem does not apply here, as the marginal $\mu_0$ is a Dirac mass at $x=0$ and not absolutely continuous with respect to Lebesgue (the other conditions from section 2 are easy to verify).  However, Lemma \ref{ineq} implies that for any $\epsilon >0$ and any curve $\tilde{x}(t)$we have 

\begin{equation*}
 h(\int_{\epsilon}^1 \tilde{x}_tdt) - \int_{\epsilon}^1g_t(\tilde{x}_t)dt \leq  h(\int_{\epsilon}^1 x_{\lambda}(t)dt) - \int_{\epsilon}^1g_t(x_{\lambda}(t)) dt
\end{equation*}

Taking the limit as $\epsilon \rightarrow 0$ yields:

\begin{equation*}
 h(\int_0^1 \tilde{x}_t)dt - \int_0^1g_t(\tilde{x}_t)dt \leq  h(\int_0^1 x_{\lambda}(t)dt) - \int_0^1g_t(x_{\lambda}(t)) dt
\end{equation*}

We can then prove, just as in Section 3, that for any process $Y_t \in \Pi$, we have

\begin{equation*}
E( h(\int_0^1 Y_tdt))\leq E(h(\int_0^1 X^{opt}_tdt))
\end{equation*}
\end{dirac}
\begin{proof}
Let $h$ be the exponential function.  The curves $x_{\lambda}(t)$ here satisfy

\begin{equation*}
\lambda = \int_{-\infty}^{x_{\lambda}(t)} \frac{e^{\frac{-y^2}{2t}}}{\sqrt{2\pi t}}dy = \int_{-\infty}^{\frac{x_{\lambda}(t)}{\sqrt{t}}} \frac{e^{\frac{-y^2}{2}}}{\sqrt{2\pi}}dy 
\end{equation*}
which implies $x_{\lambda}(t)=\sqrt{t} z_{\lambda}$, where $\lambda = $erf$(z_{\lambda})$. Therefore, $\int_0^sx_{\lambda}(t)dt = \frac{2}{3}z_{\lambda}s^{\frac{3}{2}}$. Our expectation becomes:

\begin{eqnarray*}
E(h(\int_0^s X_t^{opt}dt))&=&\int_0^1 e^{\frac{2}{3}z_{\lambda}s^{\frac{3}{2}}}d\lambda\\
 &=& \frac{1}{\sqrt{2\pi}}\int_{-\infty}^{\infty}e^{\frac{2}{3}zs^{\frac{3}{2}}}e^{\frac{-z^2}{2}} dz\\
&=& \frac{1}{\sqrt{2\pi}}\int_{-\infty}^{\infty}e^{-\frac{1}{2}[(z-\frac{2s^{\frac{3}{2}}}{3})^2-4\frac{s^3}{9}]}dz\\
&=& \frac{e^{\frac{2s^3}{9}}}{\sqrt{2\pi}}\int_{-\infty}^{\infty}e^{-\frac{1}{2}(z-\frac{4s^{\frac{3}{2}}}{3})^2}dz\\
&=&e^{\frac{2s^3}{9}}
\end{eqnarray*}

Now, as $W_t$ is Brownian motion, it is well known that $Z_s=\int_{0}^sW_tdt$ is normally distributed, with mean zero and variance $\frac{s^3}{3}$.  Therefore, we have that

\begin{eqnarray*}
E(e^{\int_{0}^sW_tdt})&=& \sqrt{\frac{3}{2\pi s^3}}\int_{-\infty}^{\infty} e^{z}e^{-\frac{3z^2}{2s^3}}dz\\
&=&e^{\frac{s^3}{6}}\sqrt{\frac{3}{2\pi s^3}}\int_{-\infty}^{\infty} e^{-\frac{3(z-\frac{s^3}{3})^2}{2s^3}}dz\\
&=&e^{\frac{s^3}{6}}
\end{eqnarray*}
Finally, we compute the upper bound from Jensen's inequality.

\begin{eqnarray*}
E(e^{\int_{0}^sW_tdt}) & \leq & E(\frac{1}{s}\int_{0}^se^{sW_t}dt)\\
&=&\frac{1}{s}\int_{0}^sE(e^{sW_t})dt\\
&=&\frac{1}{s}\int_{0}^s\int_{-\infty}^{\infty}\frac{1}{\sqrt{2\pi t}}e^{sW}e^{\frac{-W^2}{2t}}dWdt\\
&=&\frac{1}{s}\int_{0}^se^{\frac{s^2t}{2}}\int_{-\infty}^{\infty}\frac{1}{\sqrt{2\pi t}}e^{\frac{-(W-st)^2}{2t}}dWdt\\
&=&\frac{1}{s}\int_{0}^se^{\frac{s^2t}{2}}ds\\
&=&\frac{2}{s^3}[e^{\frac{s^3}{2}}-1]
\end{eqnarray*}

\end{proof}

\subsection{Upper bounds on solutions to Parabolic PDE}

Consider the partial differential equation on $(\vec{x}, t) \in \mathbb{R}^{n+1}$:
\begin{equation*}
    \frac{\partial f}{\partial t} +\sum_{i=1}^n b_i(\vec{x},t) \frac{\partial f}{\partial x_i} + \frac{1}{2} \sum_{i,j}(A^TA)_{ij}(\vec{x},t) \frac{\partial^2 f}{\partial x_ix_j} = V(\vec{x},t) f 
\end{equation*}
with terminal condition $f(\vec{x},T)=\psi(\vec{x})$.  Here, $b = b(\vec{x}, t)$ is an $\mathbb{R}^n$ valued function and $A=A(\vec{x}, t)$ an $n \times n$ matrix valued function. Equations of this type have numerous applications in mathematical finance, mathematical physics, biology, etc.  According to the Feynman-Kac formula, the solution is given by:
\begin{equation} \label{feynmankac}
f(\vec{x},s) = E[ e^{- \int_s^T V(\vec{X_t}) dt}\psi(\vec{X}_T)]
\end{equation}  
where $\vec{X}_{t}$ is the solution of the stochastic PDE
\begin{equation*}
d\vec{X} = b(\vec{X},t)dt + A(\vec{X},t)d\vec{W},\text{ }  \vec{X}_s=\vec{x}
\end{equation*}
Here $\vec{W}$ is a Weiner process on $\mathbb{R}^n$.  Suppose that $\psi > 0$.  Letting $\mu_t=$law$(V(\vec{X}_t))$ and $\nu =$law$(\psi(\vec{X}_T))$, then, if assumptions 1-6 are satisfied, Proposition \ref{reeq} tells us that:

\begin{eqnarray}\label{fkbound}
f(\vec{x},s) &\leq &E[ e^{- \int_s^T X^{opt}_t dt}Z^{opt}]\\
&=&\int_0^1 e^{- \int_s^T x_{\lambda}(t) dt}z_{\lambda}d\lambda
\end{eqnarray}

Depending on $b$, $A$ and $V$, the Feynman-Kac formula may be very difficult to evaluate analytically.  In practice, the integral (\ref{feynmankac}) is often computed numerically using Monte-Carlo techniques.  In certain cases, however, the single time marginals of the real-valued process $V(\vec{X}_t)$ can be determined.  To find an upper bound on $f(\vec{x},t)$, then, we must only compute the integral (\ref{fkbound}), which boils down to computing two, 1-dimensional integral.  Even if this integral cannot easily be calculated analytically, it is a much simpler task to compute numerically than the \emph{infinite} dimensional integral in (\ref{feynmankac}).  In the next subsection, we present a particular application in quantum physics.

\section{Application: phase space bounds in quantum physics}
In this subsection, we illustrate how our result can be used to prove refinements of two phase-space bounds in quantum physics.  A discussion of the physical significance of these bounds can be found in the book of Simon \cite{simon}.  The basic idea behind their proofs is that, as outlined by Simon, phase space bounds can be derived by applying Jensen's inequality to quantities of the form $(MK_{\infty})$.  Our general strategy here is to mimic these proofs, replacing Jensen's inequality with the result from Proposition \ref{reeq}; as a consequence of Proposition \ref{beatjen}, our phase space bounds here are tighter. 

Let $V$ be a continuous, integrable, function on $\mathbb{R}^n$ which is bounded from below.  For fixed $t$, let $W_s^t$ be a Brownian bridge process in $s$ on $\mathbb{R}^n$, starting and ending at $\vec{x}=0$, at times $s=0$ and $s=t$, respectively; that is, $W_0^t=0$ and $W_t^t=0$.

Let $h$ be the exponential function and, for each $\vec{x} \in \mathbb{R}^n$, assume that the measures $\mu_{s,t,\vec{x},V} =$law$[V(\vec{x}+W_s^t)]$ satisfy conditions 1-6\footnote{Note that, strictly speaking, the measures $\mu_{s,t,\vec{x},V}$ cannot satisfy assumption 2, as $\mu_{0,t,\vec{x},V}$ and $\mu_{t,t,\vec{x},V}$ are Dirac masses at $V(\vec{x})$; however, this can be dealt with as in the example in subsection 4.4.}.  Let $Y^{\lambda}_{s,t,\vec{x},V}$ be defined by the implicit relationship

\begin{equation*}
\lambda = \mu_{s,t,\vec{x},V}(-\infty,Y^{\lambda}_{s,t,\vec{x},V}).
\end{equation*}Let $H_0$ be $-\frac{1}{2}\Delta$, where $\Delta$ is the Laplacian.

The following classical bound on the trace of the operator $e^{-t(H_0+V)}$ is due to Golden \cite{golden}, Thompson \cite{thompson} and Symanzik \cite{sym}:

\begin{equation}\label{classtrace}
Tr(e^{-t(H_0+V)})\leq \int_{\mathbb{R}^n} \frac{1}{(2\pi t)^{n/2}}e^{-tV(\vec{x})} d\vec{x},
\end{equation}
The proof of our bound below is based on Symanzik's proof of inequality (\ref{classtrace}), outlined in the book by Simon.

\newtheorem{trace}{Theorem}[subsection]
\begin{trace} \label{trace}
\begin{equation*}
Tr(e^{-t(H_0+V)})\leq \int_{\mathbb{R}^n} \int_{0}^{1} e^{\int_{0}^{t} Y^{\lambda}_{s,t,\vec{x},V}ds}d\lambda d\vec{x},
\end{equation*}
\end{trace}

\begin{proof}
We follow the argument in Simon, using our inequality from Theorem \ref{reeq} in place of Jensen's inequality.

As in Simon, the Feynman-Kac formula implies:

\begin{equation*}
Tr(e^{-t(H_0+V)}) =  \int_{\mathbb{R}^3} E\big[ e^{-\int_0^t V(\vec{x}+W_s^t)ds}\big]d\vec{x}.
\end{equation*}
 Applying Proposition \ref{reeq} to $ E\big[ e^{-\int_0^t V(\vec{x}+W_s^t)ds}\big]$ yields the desired result. 
\end{proof}

A similar argument yields a bound on the number of bound states, that is, the number $N(V)$ of negative eigenvalues of $H_0$.  Using Jensen's inequality, this leads to Lieb's proof of the celebrated Cwickel-Lieb-Rosenbljum bound \cite{lieb}, which states that, for $n \geq 3$, $N(V)$ is bounded from above by the $L^{n/2}$ norm of the negative part of $V$:

\begin{equation}\label{cclr}
N(V) \leq a \int_{\mathbb{R}^n}|V_{-}(\vec{x})|^{n/2}d\vec{x}.
\end{equation}

Let us mention that the independent proofs of Cwikel \cite{cwikel} and Rosenbljum \cite{ros} rely on different techniques.  Other proofs of the CLR bound have since been found by Li and Yau \cite{liyau}, Fefferman \cite{feff} and Conlon \cite{conlon}.

Applying our main theorem instead leads to the bound below.

\newtheorem{clr}[trace]{Theorem}
\begin{clr}\label{clr}
For any convex function $h$ satisfying 1-6, 
 
 \begin{equation*}
 N(V) \leq  C(h) \int_0^{\infty}t^{-1}\int_{\mathbb{R}^n}\int_{0}^{1}h(\int_{0}^t Y_{s,t,\vec{x},V}^{\lambda}ds)d\lambda d\vec{x}dt.
 \end{equation*}
Here, $C(h)$ is a constant depending on $h$. 
\end{clr}
\begin{proof}
The argument of Lieb (again, reproduced by Simon) yields, for any convex function $h$ on $\mathbb{R}$

\begin{equation*}
N(V) \leq C(h) \int_0^{\infty}t^{-1}\int_{\mathbb{R}^n}  E\Big( h\big(\int_0^tV(W_s^t)ds\big)\Big) d\vec{x}dt
\end{equation*}
We apply our inequality to $E\Big( h\big(\int_0^tV(W_s^t)ds\big)\Big)$, yielding the desired result.  

\end{proof}
\newtheorem{better}{Remark}
\begin{better}
Let us note that the above holds for any convex function $h$.  Lieb's proof goes on to apply Jensen's inequality to $ h\big(\int_0^tV(W_s^t)ds\big)$ and obtains (\ref{cclr}) with
\begin{equation*}
a = C(h)\frac{1}{2\pi}\int_0^{\infty}s^{-1-n/2}h(s)ds
\end{equation*}
If $n\geq3$, it is possible to choose the function $h$ so that both $C(h)$ and the preceding integral are finite \cite{lieb}\cite{simon}.  In principal, one can then optimize over convex functions $h$ to obtain as good a constant $a$ as possible (note, however, that the optimal constant is not known).  Note, however, that for whichever function $h$ one chooses, our bound above, using the same $h$, is tighter.
\end{better}
It should be emphasized that, while our bounds appear complicated, they can be calculated for any convex function $h$ simply by determining the single time marginals of $V[\vec{x}+W_s^t]$ and calculating several (two in Theorem \ref{trace} and three in Theorem \ref{clr}) $1$-dimensional integrals and one $n$-dimensional integral, which is relatively easy to do numerically even if the form of $V$ makes it difficult analytically.  On the other hand, our bounds \textit{improve} the classical phase space bounds, as our inequality from Proposition \ref{reeq} is always tighter than Jensen's inequality, by Proposition \ref{beatjen}.
\bibliographystyle{plain}
\bibliography{biblio}

\end{document}